\documentclass{amsart}
\usepackage{amsmath,cite}
\usepackage{amsfonts}
\usepackage{amssymb,enumerate}
\usepackage{amsthm}
\usepackage[all]{xy}
\usepackage{hyperref}

\newcommand{\Kr}{\operatorname{Kr}}
\newcommand{\KR}{\operatorname{Kr_h}}
\newcommand{\SStar}{\operatorname{SStar}}
\newcommand{\Over}{\operatorname{Overr}}
\newcommand{\Zar}{\operatorname{Zar}}
\newcommand{\Spec}{\operatorname{Spec}}

\newcommand{\fa}{\frak{a}}

\newcommand{\fp}{\frak{p}}

\newtheorem{theorem}{Theorem}
\newtheorem{proposition}[theorem]{Proposition}

\newtheorem{example}{Example}
\newtheorem{remark}{Remark}
\newtheorem{lemma}{Lemma}
\newtheorem{corollary}{Corollary}

\begin{document}

\bibliographystyle{amsplain}

\date{}

\author{Parviz Sahandi}

\address{Department of Pure Mathematics, Faculty of Mathematical Sciences, University of Tabriz, Tabriz,
Iran}
\email{sahandi@ipm.ir}



\title[The space of homogeneous preserving semistar operations]{The space of homogeneous preserving semistar operations on graded domains}

\begin{abstract} Let $R=\bigoplus_{\alpha\in\Gamma}R_{\alpha}$ be a graded integral domain. In this paper we study the space of homogeneous preserving semistar operations on $R$. We show if $\star$ is a homogeneous preserving semistar operation on $R$, then $\star_a$ is also homogeneous preserving. Let $\KR(R,b)$ be the homogeneous Kronecker function ring of $R$ with respect to the $b$-operation. It is shown that the set of valuation overrings of $\KR(R,b)$, endowed with the Zariski topology, is homeomorphism to $\Zar_h(R)$, the set of gr-valuation overrings of $R$, endowed with the Zariski topology. We also show that the set $\SStar_{f,hp}(R)$ of finite type, homogeneous preserving semistar operations on $R$, endowed with the Zariski topology, is a spectral space.
\end{abstract}

\maketitle

\section{Introduction}

Star operations are introduced by Krull in \cite{k36}, that is ``special'' closure operations on nonzero fractional ideals of an integral domain $R$ with quotient field $K$. Star operations have been proven to be an essential tool in \emph{multiplicative ideal theory}, allowing one to study different classes of integral domains. In 1994, Okabe and Matsuda \cite{OM} introduced the concept of \emph{semistar operations} to extend the notion of classical \emph{star operations} as described in \cite[Section 32]{g72}. Semistar operations, thanks to a higher flexibility than star operations, permit a finer study and new classifications of special classes of integral domains (cf. also \cite{fh00, FJS, fl03, FL1,FL2,FL3,aa90,s18}).

The purpose of this paper is to study the space of homogeneous preserving semistar operations on graded integral domains $R=\bigoplus_{\alpha\in\Gamma}R_{\alpha}$ graded by an arbitrary grading torsionless monoid $\Gamma$. In Section 2, we give preliminary definitions and examples of homogeneous preserving semistar operations on $R$. In Section 3, we use the homogeneous Kronecker function rings, to show that if $\star$ is a homogeneous preserving semistar operation on $R$, then $\star_a$ is homogeneous preserving. Then we able to show that, the space of gr-valuation overrings of $R$, denoted by $\Zar_h(R)$, endowed with the Zariski topology is a spectral space. In Section 4, we show that the space $\SStar_{f,hp}(R)$ of finite type homogeneous preserving semistar operations endowed with Zariski topology is a spectral space.


Let $\Gamma$ be a nonzero torsionless commutative cancellative monoid (written
additively), and $\langle \Gamma \rangle = \{a - b \mid a,b \in \Gamma\}$ be the
quotient group of $\Gamma$; so $\langle \Gamma \rangle$ is a torsionfree abelian group.
It is known that a cancellative monoid $\Gamma$ is torsionless
if and only if $\Gamma$ can be given a total order compatible with the monoid
operation \cite[page 123]{n68}. By a $(\Gamma$-)graded integral domain  $R =\bigoplus_{\alpha \in \Gamma}R_{\alpha}$,
we mean an integral domain graded by $\Gamma$.
That is, each nonzero $x \in R_{\alpha}$ has degree $\alpha$, i.e., deg$(x) = \alpha$,  and
deg$(0) = 0$. Thus, each nonzero $f \in R$ can be written uniquely as $f = x_{\alpha_1} + \dots + x_{\alpha_n}$ with
deg$(x_{\alpha_i}) = \alpha_i$ and $\alpha_1 < \cdots < \alpha_n$.

An element $x \in R_{\alpha}$ for $\alpha \in \Gamma$ is said to be {\em homogeneous}. Let $H$ be the set of nonzero homogeneous elements of $R$. Then $H$ is a saturated multiplicative set of $R$, $H = \bigcup_{\alpha \in \Gamma}(R_{\alpha} \setminus \{0\})$, and $R_H$ is a $\langle \Gamma \rangle$-graded integral domain with $(R_H)_{\alpha}:=\{\frac{a}{b}\mid a\in R_{\beta}, 0\neq b\in R_{\gamma},$ and $\beta-\gamma=\alpha\}$, whose nonzero homogeneous elements are units. We say that $R_H$ is the {\em homogeneous quotient field} of $R$. An integral ideal $I$ of $R$ is said to be homogeneous if $I=\bigoplus_{\alpha\in\Gamma}(I\cap R_{\alpha})$. A fractional ideal $I$ of $R$ is \emph{homogeneous} if $sI$ is an integral homogeneous ideal of $R$ for some $s\in H$ (thus $I\subseteq R_H$). A homogeneous ideal of $R$ is called a {\em homogeneous maximal ideal} if it is maximal among proper homogeneous ideals of $R$. An overring $A$ of $R$ is a {\em homogeneous overring} of $R$ if $A = \bigoplus_{\alpha \in \langle \Gamma \rangle}(A \cap (R_H)_{\alpha})$; so $A$ is a $\langle \Gamma \rangle$-graded integral domain such that $R \subseteq A \subseteq R_H$. Clearly, if $\Lambda = \{\alpha \in \langle \Gamma \rangle \mid A \cap (R_H)_{\alpha} \neq \{0\}\}$, then $\Lambda$ is a torsionless commutative cancellative monoid such that $\Gamma \subseteq \Lambda \subseteq \langle \Gamma \rangle$ and $A = \bigoplus_{\alpha \in \Lambda}(A \cap (R_H)_{\alpha})$.

We say that $R$ is a graded-valuation domain (gr-valuation domain) if either $x\in R$ or $x^{-1}\in R$ for every nonzero homogeneous element $x\in R_H$. The notion of gr-valuation domains for $\mathbb{Z}$-graded integral domains was introduced by Johnson \cite{j79}. He then studied $\Gamma$-graded integral domains
in \cite{j83}, where he showed that the integral closure of $R$ is the intersection of all the gr-valuation (homogeneous) overrings of $R$ \cite[Theorem 2.10]{j83}. The gr-valuation domain $R$ has a unique homogeneous maximal ideal $M$, $R_M$ is a valuation domain and $R_M\cap R_H=R$ (see \cite[Lemma 4.3]{s14} and \cite[Theorem 2.3]{aac17}).

For more on graded integral domains and their divisibility properties, see \cite{AA2}, \cite{n68}.

\section{Semistar operations on graded integral domains}

Let $\Gamma$ be a (nonzero) torsionless commutative cancellative monoid, $R=\bigoplus_{\alpha\in\Gamma}R_{\alpha}$ be an integral domain graded by $\Gamma$, and $H$ be the set of nonzero homogeneous elements of $R$ with quotient field $K$.

Let $\overline{\mathcal{F}}(R)$ denote the set of all nonzero $R$-submodules of $K$, $\mathcal{F}(R)$ be the set of all nonzero fractional ideals of $R$, and $f(R)$ be the set of all nonzero finitely generated fractional ideals of $R$. Obviously, $f(R)\subseteq\mathcal{F}(R)\subseteq\overline{\mathcal{F}}(R)$. As
in \cite{OM}, a {\it semistar operation on} $R$ is a map $\star:\overline{\mathcal{F}}(R)\rightarrow\overline{\mathcal{F}}(R)$, $E\mapsto E^{\star}$, such that, for all $0\neq x\in K$, and for all $E, F\in\overline{\mathcal{F}}(R)$, the following four properties
hold:
\begin{description}
  \item[$\star_1$] $(xE)^{\star}=xE^{\star}$;
  \item[$\star_2$] $E\subseteq F \Rightarrow E^{\star}\subseteq F^{\star}$;
  \item[$\star_3$] $E\subseteq E^{\star}$;
  \item[$\star_4$] $E^{\star\star}:=(E^{\star})^{\star}=E^{\star}$.
\end{description}
A semistar operation $\star$ is called a \emph{(semi)star operation on $R$}, if $R^{\star}=R$.

The set $\SStar(R)$ of all semistar operations on $R$ is endowed with a natural partial order. If $\star_1$ and $\star_2$ are semistar operations on $R$, one says that $\star_1\leq\star_2$ if $E^{\star_1}\subseteq E^{\star_2}$ for each $E\in\overline{\mathcal{F}}(R)$. This is equivalent to saying that $(E^{\star_1})^{\star_2}=E^{\star_2}=(E^{\star_2})^{\star_1}$ for each $E\in\overline{\mathcal{F}}(R)$ (cf. \cite[Lemma 16]{OM}). Recall that an ideal $I$ of $R$ is a \emph{quasi-$\star$-ideal} if $I^{\star}\cap R=I$, and is a \emph{quasi-$\star$-prime}, if it is prime and quasi-$\star$-ideal.

We say that a semistar operation $\star\in\SStar(R)$ is \emph{homogeneous preserving} if $\star$ sends homogeneous fractional ideals of $R$ to homogeneous ones (see \cite{s14} and \cite{s18}). In particular if $\star$ is a homogeneous preserving semistar operation on $R$, we have $R^{\star}\subseteq R_H$. It is shown in \cite[Lemma 2.1]{s14} that, if $\star$ is a finite type homogeneous preserving semistar operation on $R$ (see Example \ref{ex}(c)) such that $R^{\star}\subsetneq R_H$, then each proper homogeneous quasi-$\star$-ideal of $R$ is contained in a homogeneous quasi-$\star$-prime ideal of $R$.

The set of homogenous preserving semistar operations $\star$ on $R$ is denoted by $\SStar_{hp}(R)$. Here we note that, recently the notion of homogeneous star operations was defined in \cite[Definition 3.4]{FO12} on homogeneous ideals of the $\mathbb{Z}$-graded domain $S:=L[X_0,\ldots,X_n]$ where $L$ is a field, satisfying the same axioms of a star operation \cite[Section 32]{g72}. In particular, a homogeneous preserving (semi)star operation on $S$, is a homogeneous star operation in the sense of \cite{FO12}.

In the following example we collect various homogeneous preserving semistar operations on $R$.

\begin{example}\label{ex}{\em Let $R=\bigoplus_{\alpha\in\Gamma}R_{\alpha}$ be a
graded integral domain.
\begin{itemize}
  \item[(a)] $d:=d_R$ denotes the identity semistar operation on $R$, is homogeneous preserving.
  \item[(b)] Let $T$ be an overring of $R$, we denote by $\star_{\{T\}}$ the semistar operation on $R$ by setting $F^{\star_{\{T\}}}:=FT$, for any $F\in \overline{\mathcal{F}}(R)$. It is easy to see that if $T$ is a homogeneous overring of $R$, then $\star_{\{T\}}$ is homogeneous preserving.
  \item[(c)] Let $\star$ be a semistar operation on $R$. We can associate a new semistar operation $\star_f$ on $R$ by setting $$E^{\star_f}:=\bigcup\{F^{\star}\mid F\in f(R)\text{ }and\text{ }F\subseteq E\},$$ for any $E\in \overline{\mathcal{F}}(R)$. We call the semistar operation $\star_f$ the finite type semistar operation associated to $\star$. We call $\star$ a \emph{semistar operation of finite type} if $\star=\star_f$. Note that $(\star_f)_f=\star_f$, and thus ${\star_f}$ is a semistar operation of finite type. The set of finite type semistar operations on $R$ is denoted by $\SStar_f(R)$. By \cite[Lemma 2.4]{s14}, if $\star$ is homogeneous preserving, then also $\star_f$ is homogeneous preserving.
  \item[(d)] $v$ denotes the divisorial semistar operation on $R$, defined by $F^v:=(R:(R:F))$, for any $F\in \overline{\mathcal{F}}(R)$. By \cite[Proposition 2.5]{AA2}, the $v$-operation is homogeneous preserving. The finite type semistar operation associated to $v$ is usually denoted by $t$. Hence the $t$-operation is homogeneous preserving by part (c). 
  \item[(e)] Let $\mathcal{S}$ be a nonempty collection of elements of $\SStar_{hp}(R)$. Then $\bigwedge(\mathcal{S})$ is homogeneous preserving semistar operation on $R$ defined by setting $$F^{\bigwedge(\mathcal{S})}:=\bigcap\{F^{\star}\mid\star\in \mathcal{S}\}\text{ for any }F\in \overline{\mathcal{F}}(R).$$ It is easy to see that $\bigwedge(\mathcal{S})$ is the infimum of $\mathcal{S}$ in the partially ordered set $(\SStar_{hp}(R),\leq)$. Moreover, the semistar operation $$\bigvee(\mathcal{S}):=\bigwedge(\{\sigma\in\SStar_{hp}(R)\mid\sigma\geq\star,\text{ for any }\star\in \mathcal{S}\})$$ is the supremum of $\mathcal{S}$ in the partially ordered set $(\SStar_{hp}(R),\leq)$ is homogeneous preserving.
  \item[(f)] Let $Y$ be a nonempty collection of overrings of $R$. Then $\wedge_Y:=\bigwedge(\{\star_{\{B\}}\mid B\in Y\})$ is a semistar operation on $R$. In other words, the semistar operation $\wedge_Y$ is defined by setting $$F^{\wedge_Y}:=\bigcap\{FB\mid B\in Y\}\text{ for any }F\in \overline{\mathcal{F}}(R).$$ It is easy to see that if $Y$ consists of homogeneous overrings of $R$, then $\wedge_Y$ is homogeneous preserving.
  \item[(g)] Let $X$ be a nonempty collection of homogeneous prime ideals of $R$. Then the semistar operation $h_X:=\wedge_{\{R_{H\setminus \frak{p}}\mid\frak{p}\in X\}}$ is homogeneous preserving.
  \item[(h)] We say that a semistar operation $\star$ is \emph{stable} if $(E\cap F)^{\star}=E^{\star}\cap F^{\star}$, for any $E, F\in \overline{\mathcal{F}}(R)$. Since localization commutes with finite intersections, any semistar operation in part (g) is stable.
  \item[(i)] Let $\star$ be semistar operation in $R$. For any $F\in \overline{\mathcal{F}}(R)$ set $$F^{\widetilde{\star}}:=\bigcup\{(F:\fa)\mid \fa\text{ is a finitely generated ideal of }R\text{ and }\fa^{\star}=R^{\star}\}.$$ Then $\widetilde{\star}$ is a stable and of finite type semistar operation on $R$ and $\widetilde{\star}\leq\star_f$ (see \cite{FL2}). We say that $\widetilde{\star}$ is \emph{the stable semistar operation associated to $\star$}. By the following Proposition \ref{pp}, $\widetilde{\star}$ is homogeneous preserving for each semistar operation $\star$ on $R$ such that $R^{\widetilde{\star}}\subseteq R_H$.

  \item[(j)] Let $\star$ be a semistar operation on an integral domain $R$. We say that $\star$ is an \emph{\texttt{e.a.b.} (endlich arithmetisch brauchbar) semistar operation} of $R$ (after Krull and Gilmer, see \cite[Remark 1]{fl09}) if, for all $E, F, G\in f(R)$, $(EF)^{\star}\subseteq(EG)^{\star}$ implies that $F^{\star}\subseteq G^{\star}$ (\cite[Definition 2.3 and Lemma 2.7]{FL2}). We can associate  to any semistar operation $\star$ on $R$, an \texttt{e.a.b.} semistar operation of finite type $\star_a$ on $R$, called the \emph{\texttt{e.a.b.} semistar operation associated to $\star$}, defined as follows for each $F\in f(R)$ and for each $E\in\overline{\mathcal{F}}(R)$:
\begin{align*}
F^{\star_a}:=&\bigcup\{((FH)^{\star}:H^{\star})\mid  H\in f(R)\},\\[1ex]
E^{\star_a}:=&\bigcup\{F^{\star_a}\mid  F\subseteq E, F\in f(R)\}
\end{align*}
\cite[Definition 4.4 and Proposition 4.5]{FL2}. It is known that $\star_f\leq\star_a$ \cite[Proposition 4.5(3)]{FL2}. Obviously $(\star_f)_a=\star_a$. Moreover, when $\star=\star_f$, then $\star$ is \texttt{e.a.b.} if and only if $\star=\star_a$ \cite[Proposition 4.5(5)]{FL2}. By the following Corollary \ref{c}, $\star_a\in\SStar_{hp}(R)$ for each homogeneous preserving semistar operation $\star$ on $R$.
\end{itemize}}
\end{example}

For the convenience of the reader we state \cite[Proposition 2.3]{s14} here correcting a typo in its proof. If $a\in R$, we denote by $C(a)$ the homogeneous ideal of $R$ generated by the homogeneous components of $a$.

\begin{proposition}\label{pp} Let $R=\bigoplus_{\alpha\in\Gamma}R_{\alpha}$ be a graded integral domain, and $\star$ be a semistar operation on $R$ such that $R^{\widetilde{\star}}\subseteq R_H$. Then, $\widetilde{\star}$ is homogeneous preserving. 
\end{proposition}

\begin{proof} Let $I$ be a homogenous ideal of $R$. To show that $I^{\widetilde{\star}}$ is homogeneous let $f\in I^{\widetilde{\star}}$. Then $fJ\subseteq I$ for some finitely generated ideal $J$ of $R$ such that $J^{\star}=R^{\star}$. Suppose that $J=(g_1,\cdots,g_n)$. Using \cite[Lemma 1.1(1)]{ac13}, there is an integer $m\geq2$ such that $C(g_i)^mC(f)=C(g_i)^{m-1}C(fg_i)$ for all $i=1,\cdots,n$. Since $I$ is a homogeneous ideal and $fg_i\in I$, we have $C(fg_i)\subseteq I$. Thus we have $C(g_i)^mC(f)\subseteq I$. Since $J\subseteq(C(g_1)+\cdots+C(g_n))$ we have $(C(g_1)+\cdots+C(g_n))^{\star}=R^{\star}$. Put $J_0:=(C(g_1)+\cdots+C(g_n))^{nm}$. Thus $J_0$ is a finitely generated homogeneous ideal of $R$ such that $J^{\star}_0=R^{\star}$ and that $C(f)J_0\subseteq I$, hence $C(f)\subseteq I^{\widetilde{\star}}$. This means that $I^{\widetilde{\star}}$ is homogeneous. Now assume that $E$ is a homogeneous fractional ideal of $R$. Then there exists $0\neq x\in H$ such that $I:=xE\subseteq R$. Hence $xE^{\widetilde{\star}}=I^{\widetilde{\star}}\subseteq R^{\widetilde{\star}}\subseteq R_H$ is homogeneous. Therefore $E^{\widetilde{\star}}=x^{-1}(xE^{\widetilde{\star}})$ is a homogeneous fractional ideal.
\end{proof}

\section{The space of $\texttt{e.a.b.}$ semistar operations of finite type on a graded domain}

Let $\Gamma$ be a (nonzero) torsionless commutative cancellative monoid, $R=\bigoplus_{\alpha\in\Gamma}R_{\alpha}$ be an integral domain graded by $\Gamma$, $H$ be the set of nonzero homogeneous elements of $R$, and $K$ is the quotient field of the integral domain $R$.

The following example shows that \texttt{e.a.b.} semistar operations on graded domains are not homogeneous preserving in general.

\begin{example} Let $R=D[X]$ be the polynomial ring over an integral domain $D$ with quotient field $L$, then $R$ is a graded integral domain with $\deg(aX^n)=n$ for every $0\neq a\in D$. Then $R_H=L[X,X^{-1}]$ and let $V$ be a valuation overring of $R_H$. Hence $\star:=\star_{\{V\}}$ is an \texttt{e.a.b.} semistar operation on $R$ but it is not homogeneous preserving.
\end{example}

We aim to show that if $\star$ is a homogeneous preserving, then so is $\star_a$. We begin by recalling the definition of Kronecker function rings with respect to $\star$. Let $\star$ be a semistar operation on $R$. The \emph{Kronecker function ring with respect to $\star$} is defined as follows (see, \cite[Section 32]{g72}, \cite{fl03}, \cite{FL2} and \cite{FL3}): $$\mbox{Kr}(R,\star):=\left\{\frac{f}{g}\mid  \begin{array}{l} f,g\in R[X], g\neq0,\text{ and there is }0\neq h\in R[X]\\\text{ such that } c(f)c(h)\subseteq(c(g)c(h))^{\star} \end{array} \right\},$$ where $c(f)$ is the ideal of $R$ generated by the coefficients of $f\in R[X]$. It is known that $\Kr(R,\star)$ is a B\'{e}zout domain, $\Kr(R,\star)=\Kr(R,\star_a)$ and that $F^{\star_a}=F\Kr(R,\star)\cap K$, for each $F\in \mathcal{F}(R)$, \cite[Proposition 4.1]{fl03}.

In working with graded domains, the \texttt{e.a.b.} condition is too strong, as we need the cancelation condition only for homogeneous ideals. We say that a semistar operation $\star$ on $R$ is \emph{graded-\texttt{e.a.b.}}, if for all finitely generated homogeneous fractional ideals $E, F, G$, $(EF)^{\star}\subseteq(EG)^{\star}$ implies that $F^{\star}\subseteq G^{\star}$. In particular, every \texttt{e.a.b.} semistar operation on $R$, is a graded-\texttt{e.a.b.} semistar operation. If $V$ is a gr-valuation overring of $R$, then the semistar operation $\star_{\{V\}}$ is a graded-\texttt{e.a.b.} semistar operation on $R$. More generally, if $Y$ is a nonempty collection of gr-valuation overrings of $R$, then $\wedge_Y$ is a graded-\texttt{e.a.b.} semistar operation on $R$. Note that $\wedge_Y$ is a homogeneous preserving semistar operation on $R$.

For a polynomial $f=f_0+f_1X+\cdots+f_nX^n\in R[X]$, we define the \emph{homogeneous content ideal} of $f$ by $\mathcal{A}_f:=\sum_{i=0}^nC(f_i)$. It is easy to see that $\mathcal{A}_f$ is a homogeneous finitely generated ideal of $R$, and that if $R$ has trivial grading, then $\mathcal{A}_f=c(f)$. The homogeneous content ideal has similar behavior as usual content ideal $c(f)$. For example, the Dedekind-Mertens lemma holds for homogeneous content ideal $\mathcal{A}_f$ \cite[Proposition 2.1]{s18}.

It is possible to define a homogeneous Kronecker function ring using the homogeneous content ideal $\mathcal{A}_f$ as in \cite{s18}. Assume that $\star$ is a semistar operation on $R$. Then we define the \emph{homogeneous Kronecker function ring with respect to $\star$} by the following: $$\KR(R,\star):=\left\{\frac{f}{g}\mid \begin{array}{l} f,g\in R[X], g\neq0,\text{ and there is }0\neq h\in R[X]\\\text{ such that } \mathcal{A}_f\mathcal{A}_h\subseteq(\mathcal{A}_g\mathcal{A}_h)^{\star} \end{array} \right\}.$$ Note that $\KR(R,\star)$ was denoted by $\mbox{KR}(R,\star)$ in \cite{s18}. In the next proposition we collect several properties of $\KR(R,\star)$ that will be useful in the following.

\begin{proposition}\label{Krr} Assume that $\star$ is a graded-\texttt{e.a.b.} semistar operation on $R$. Then
$$
\KR(R,\star)=\left\{\frac{f}{g}\mid \begin{array}{l} f,g\in
R[X],\text{ }g\neq0,\text{ and }\mathcal{A}_f\subseteq \mathcal{A}_g^{\star} \end{array}
\right\}.
$$
and
\begin{enumerate}
\item[(1)] $\KR(R,\star)$ is an integral domain with quotient field $K(X)$.
\item[(2)] $\KR(R,\star)$ is a B\'{e}zout domain.
\item[(3)] $F^{\star}=F\KR(R,\star)\cap R_H$ for every nonzero finitely
generated homogeneous ideal $F$ of $R$.
\item[(4)] $\KR(R,\star)=\KR(R^{\star},\alpha(\star))$, where $\alpha(\star)$ is the restriction of $\star$ to $\overline{\mathcal{F}}(R^{\star})$.
\end{enumerate}
If $\star$ is homogeneous preserving (not necessarily graded-\texttt{e.a.b.}), then
\begin{enumerate}
\item[(5)] $\KR(R,\star)=\KR(R,\star_a)$.
\end{enumerate}
\end{proposition}
\begin{proof} For (1), (2), and (3) see \cite[Lemma 2.9]{s18} and (5) follows from \cite[Theorem 2.10]{s18}. For (4) follow the proof of \cite[Proposition 3.3]{FL1}.
\end{proof}

We note here that, for a homogeneous preserving semistar operation $\star$ on $R$, the equality $\KR(R,\star)=\Kr(R,\star)$ \emph{does not hold} in general, see \cite[Example 3.9]{s18}. However $\KR(R,\star)$ is equal to a (usual) Kronecker function ring $\Kr(R,\star^{\prime})$ for some semistar operation $\star^{\prime}$ on $R$ (see Remark \ref{halter}(2)). Also we observe that, if $R$ is a gr-valuation domain, with homogeneous maximal ideal $M$, we have that $R_M$ is a valuation domain. Then, $$\KR(R,d_R)=R[X]_{M[X]}=R_M(X)=\Kr(R_M,d_{R_M})$$ by \cite[Corollary 3.8]{s18} and \cite[Thorem 33.4]{g72}.

\begin{remark}\label{halter} (1) Let $K$ be a field, $X$ an indeterminate over $K$, and $S$ a subring of $K(X)$. We call $S$ a \emph{$K$-function ring} (after Halter-Kock \cite{h03}), if $X$ and $X^{-1}$ belongs to $S$, and for each nonzero polynomial $f\in K[X]$, $f(0)\in f(X)S$. If $S:=\KR(R,\star_a)$, then it can be seen easily from the definition that $S$ is a $K$-function ring.

(2) Assume that $\star$ is a graded-\texttt{e.a.b.} semistar operation on $R$. Let $S:=\KR(R,\star)$. Since $S$ is a $K$-function ring, we have by \cite[Theorem 3.11]{FL2} that the mapping $$E\mapsto E^{\star_S}:=\bigcup\{FS\cap K\mid F\in f(R), F\subseteq E\}\text{ for each }E\in \overline{\mathcal{F}}(R),$$ is a semistar operation of finite type on $R$ such that $S=\Kr(R,\star_S)$ is the usual Kronecker function ring.

(3) Assume that $\star$ is a homogeneous preserving semistar operation on $R$.  Then $\Kr(R,\star)\subseteq\KR(R,\star)$. Indeed, if $\frac{f}{g}\in\Kr(R,\star)$, then $c(f)c(h)\subseteq (c(g)c(h))^{\star}$ for some $0\neq h\in R[X]$. Hence $c(f)c(h)\subseteq (\mathcal{A}_g\mathcal{A}_h)^{\star}$, and since $(\mathcal{A}_g\mathcal{A}_h)^{\star}$ is a homogeneous ideal of $R$, we have $\mathcal{A}_f\mathcal{A}_h\subseteq (\mathcal{A}_g\mathcal{A}_h)^{\star}$. Thus $\frac{f}{g}\in\KR(R,\star)$.
\end{remark}

\begin{lemma}\label{cap} Let $Y$ be a nonempty collection of gr-valuation overrings of $R$. Then $$\KR(R,\wedge_Y)=\bigcap_{V\in Y}\KR(R,\wedge_{\{V\}})=\bigcap_{V\in Y}V_{M_V}(X),$$ where $M_V$ is the homogeneous maximal ideal of $V$.
\end{lemma}
\begin{proof} The first equality follows easily from the following observation for $f,g\in R[X]\setminus\{0\}$, we have $$\mathcal{A}_f\subseteq\mathcal{A}_g^{\wedge_Y}\Leftrightarrow\mathcal{A}_f\subseteq\mathcal{A}_g^{\wedge_{\{V\}}}\text{ for every } V\in Y.$$ For the second equality note, for each $V\in Y$, one has $$\KR(R,\wedge_{\{V\}})=\KR(V,d_V)=V[X]_{M_V[X]}=V_{M_V}(X),$$ by Proposition \ref{Krr}(4) and \cite[Corollary 3.8]{s18}.
\end{proof}

Similar to \cite[Page 2099]{fl09}, we can associate  to any semistar operation $\star$ on $R$, a graded-\texttt{e.a.b.} semistar operation of finite type $\star_{a(h)}$ on $R$, define as follows. For each $F\in f(R)$ and each $E\in\overline{\mathcal{F}}(R)$:
\begin{align*}
F^{\star_{a(h)}}:=&\bigcup\{((FH)^{\star}:H^{\star})\mid  H\in f_h(R)\},\\[1ex]
E^{\star_{a(h)}}:=&\bigcup\{F^{\star_{a(h)}}\mid  F\subseteq E, F\in f(R)\},
\end{align*}
where $f_h(R)$ is the set of finitely generated homogeneous fractional ideals of $R$.

\begin{proposition}\label{b}{\em Let $\star$ be a semistar operation on $R=\bigoplus_{\alpha\in\Gamma}R_{\alpha}$. Then
\begin{itemize}
\item[(1)] $\star_{a(h)}$ is a graded-\texttt{e.a.b.} semistar operation on $R$.
\item[(2)] $\star_f\leq\star_{a(h)}\leq\star_a$.
\item[(3)] $(\star_{a(h)})_a=\star_a$.
\end{itemize}}
\end{proposition}
\begin{proof} For (1), follow the proof of \cite[Proposition 4.5(1) and (2)]{FL2}. For (2), assume that $F\in f(R)$ and $x\in F^{\star}$. Then $x\in(F^{\star}:R)\subseteq F^{\star_{a(h)}}$. This means that $\star_f\leq\star_{a(h)}$. The second inequality in (2), follows from the definition. Since $\star_f\leq\star_{a(h)}\leq\star_a$ by (2), we have $\star_a=(\star_f)_a\leq(\star_{a(h)})_a\leq(\star_a)_a=\star_a$ by \cite[Proposition 4.5(4) and (8)]{FL2}. Hence $(\star_{a(h)})_a=\star_a$ and the proof of (3) completes.
\end{proof}

\begin{proposition}\label{b2}{\em Let $\star$ be a homogeneous preserving semistar operation on $R=\bigoplus_{\alpha\in\Gamma}R_{\alpha}$. Then
\begin{itemize}
\item[(1)] $\star_{a(h)}$ is homogeneous preserving.
\item[(2)] $\KR(R,\star_{a(h)})=\KR(R,\star_a)$.
\item[(3)] For each $F\in f_h(R)$ we have $$F^{\star_a}=\bigcup\{((FH)^{\star}:H^{\star})\mid  H\in f_h(R)\}=F^{\star_{a(h)}}.$$
\end{itemize}}
\end{proposition}
\begin{proof} For (1), assume that $F$ is a homogeneous fractional ideal of $R$ and $f\in F^{\star_{a(h)}}$. Then there exists an element $H\in f_h(R)$ such that $f\in((FH)^{\star}:H^{\star})$. Since $\star$ is homogeneous preserving, we have $((FH)^{\star}:H^{\star})$ is a homogeneous fractional ideal by \cite[Proposition 2.5]{AA2}. Hence $C(f)\subseteq((FH)^{\star}:H^{\star})\subseteq F^{\star_{a(h)}}$. Therefore $\star_{a(h)}$ is homogeneous preserving. For (2) note that $\star_{a(h)}$ is homogeneous preserving by (1). Hence by Propositions \ref{Krr}(5) and \ref{b}(3) respectively we have $\KR(R,\star_{a(h)})=\KR(R,(\star_{a(h)})_a)=\KR(R,\star_{a})$. Now (3) follows from (2) and Proposition \ref{Krr}(3).
\end{proof}

\begin{corollary}\label{c} If $\star$ is a homogeneous preserving semistar operation on $R=\bigoplus_{\alpha\in\Gamma}R_{\alpha}$, then $\star_a$ is also homogeneous preserving.
\end{corollary}
\begin{proof} Since $\star_a$ is of finite type, it is enough to show that $F^{\star_a}$ is homogeneous for each $F\in f_h(R)$. But this is the case because, $F^{\star_a}=F^{\star_{a(h)}}$ and $F^{\star_{a(h)}}$ is homogeneous by Proposition \ref{b2}.
\end{proof}

Let $\Zar(R):=\{V\mid V$ is a valuation overring of $R\}$ be equipped with the Zariski topology, i.e., the topology having, as subbasic open subspaces, the subsets $\Zar(R[x])$ for $x$ varying in $K$. The set $\Zar(R)$, endowed with the Zariski topology, is often called the Riemann-Zariski space of $R$ \cite[Chapter VI, Section 17, p. 110]{zs75}. It is well known that every finite type \texttt{e.a.b.} semistar operation of $R$ is of the form $\wedge_Y$, for $\emptyset\neq Y\subseteq\Zar(R)$ \cite[Corollary 5.2]{FL2}. In particular the $b$-operation of Krull, defined by $b:=b_R:=\wedge_{\Zar(R)}$ is an \texttt{e.a.b.}
semistar operation of finite type on $R$.

\begin{corollary} The $b:=b_R$-operation is homogeneous preserving.
\end{corollary}
\begin{proof} Note that $b=d_a$ \cite[Lemma 1]{fp11}, where $d$ is the identity operation on $R$.
\end{proof}

\begin{remark}\label{a}{\em Let $R=\bigoplus_{\alpha\in\Gamma}R_{\alpha}$ be a graded integral domain.
\begin{itemize}
\item[(1)] Let $\bar{R}$ be the integral closure (in its quotient field) of $R$. Then $R^b=\bar{R}$ is a homogeneous overring of $R$ (c.f. \cite[Theorem 2.10]{j83}).
\item[(2)] Let $T$ be a homogeneous overring of $R$, and $b(T):=\wedge_{\Zar(T)}$ be the $b$-operation on $T$. Then it is homogeneous preserving. Therefore $b(T)$ is a homogeneous preserving semistar operation on $R$.
\end{itemize}}
\end{remark}

We recall the notion of $\star$-valuation overring (a notion due essentially to P. Jaffard \cite[page 46]{Jaf}). For a domain $D$ and a semistar operation $\star$ on $D$, we say that a valuation overring $V$ of $D$ is a \emph{$\star$-valuation overring of $D$} provided $F^{\star}\subseteq FV$, for each $F\in f(D)$. As a graded analog we say that a gr-valuation overring $V$ of $R=\bigoplus_{\alpha\in\Gamma}R_{\alpha}$ is a \emph{gr-$\star$-valuation overring of $R$} provided $F^{\star}\subseteq FV$, for each $F\in f_h(R)$. In particular, each gr-valuation overring of $R$ is a gr-$b_R$-valuation overring.

Now using Proposition \ref{b2}, we can prove the following result, which is the graded analogue of \cite[Proposition 3.3]{FL1}.

\begin{proposition}\label{v} Let $\star$ be a homogeneous preserving semistar operation on $R=\bigoplus_{\alpha\in\Gamma}R_{\alpha}$ and $V$ be a gr-valuation overring of $R$.  Then $V$ is a gr-$\star$-valuation overring of $R$ if and only if $V$ is a gr-$\star_a$-valuation overring of $R$.
\end{proposition}
\begin{proof} $(\Rightarrow)$ Assume $V$ is a gr-$\star$-valuation overring of $R$. Let $F, H\in f_h(R)$. There is a homogeneous element $x\in H$ such that $HV=xV$. Then $((FH)^{\star}:H^{\star})=((FH)^{\star}:H)\subseteq(FHV:H)\subseteq(FxV:xR)=(FV:R)=FV$. Thus $F^{\star_a}=F^{\star_{a(h)}}\subseteq FV$ by Proposition \ref{b2}(3).

$(\Leftarrow)$ Since $\star\leq\star_a$, by \cite[Proposition 4.5]{FL2}, if $V$ is a gr-$\star_a$-valuation overring of $R$, then $F^{\star}\subseteq F^{\star_a}\subseteq FV$ for each $F\in f_h(R)$. Hence $V$ is a gr-$\star$-valuation overring of $R$.
\end{proof}

The following result is a graded analog of \cite[Theorem 3.5]{FL1}.

\begin{theorem}\label{*val} Let $\star$ be a homogeneous preserving semistar operation on $R=\bigoplus_{\alpha\in\Gamma}R_{\alpha}$. Then $V$ is a gr-$\star$-valuation overring of $R$ if and only if there exists a valuation overring $W$ of $\KR(R,\star_a)$ such that $V=\bigoplus_{\alpha\in\langle\Gamma\rangle}(W\cap (R_H)_{\alpha})$.
\end{theorem}
\begin{proof} $(\Rightarrow)$ Assume that $V$ is a gr-$\star$-valuation overring of $R$ with homogeneous maximal ideal $M_V$. Let $W:=V_{M_V}(X)$ be the trivial extension of the valuation domain $V_{M_V}$. We show that $W$ is a valuation overring of $\KR(R,\star_a)$ such that $\bigoplus_{\alpha\in\langle\Gamma\rangle}(W\cap (R_H)_{\alpha})=\bigoplus_{\alpha\in\langle\Gamma\rangle}(V_{M_V}(X)\cap (R_H)_{\alpha})=\bigoplus_{\alpha\in\langle\Gamma\rangle}(V\cap (R_H)_{\alpha})=V$. Assume that $\frac{f}{g}\in\KR(R,\star_a)$. Then, we obtain that $\mathcal{A}_f\subseteq\mathcal{A}_g^{\star_a}\subseteq\mathcal{A}_gV$, since $V$ is a gr-$\star_a$-valuation overring of $R$ by Proposition \ref{v}. Hence $\frac{f}{g}\in\KR(V,\wedge_{\{V\}})=\KR(V,d_V)=V[X]_{M_V[X]}=V_{M_V}(X)=W$, where the second equality comes from \cite[Corollary 3.8]{s18}.

$(\Leftarrow)$ Assume that $W$ is a valuation overring of $\KR(R,\star)$ such that $\bigoplus_{\alpha\in\langle\Gamma\rangle}(W\cap (R_H)_{\alpha})=V$. Let $w:K(X)\to G\cup\{\infty\}$ be the valuation, associated to the valuation ring $W$. Now let $F=(a_0,\ldots,a_n)$ be a homogeneous finitely generated fractional ideal of $R$, $$f(X):=a_0+a_1X+\cdots+a_nX^n\in R_H[X],$$ and $u\in F^\star$ be a homogeneous element. Then $\mathcal{A}_u^\star=(uR)^\star\subseteq F^\star=\mathcal{A}_f^\star$, and hence we obtain that $\frac{u}{f}\in\KR(R,\star)\subseteq W$. So that $$0\leq w(\frac{u}{f})=w(u)-w(f)\leq w(u)-\inf\{w(a_0),\ldots,w(a_n)\}.$$ Assume that $\inf\{w(a_0),\ldots,w(a_n)\}=w(a_i)$ for some $0\leq i\leq n$. Then $\frac{u}{a_i}\in W\cap (R_H)_{\alpha}\subseteq V$ where $\alpha=\deg(\frac{u}{a_i})=\deg(u)-\deg{a_i}$. Hence $$u\in a_iV\subseteq(a_0,\ldots,a_n)V=FV,$$ that is $F^{\star}\subseteq FV$.
\end{proof}

\begin{remark}{\em Let $\star$ be a homogeneous preserving semistar operation on $R=\bigoplus_{\alpha\in\Gamma}R_{\alpha}$, $\mathcal{W}=\Zar(\KR(R,\star_a))$, and $\mathcal{V}=\{\bigoplus_{\alpha\in\langle\Gamma\rangle}(W\cap (R_H)_{\alpha})\mid W\in \mathcal{W}\}$. By Theorem \ref{*val} there is a bijection $\mathcal{V}\to \mathcal{W}$, $V\mapsto V_{M_V}(X)$. Using Lemma \ref{cap}, we have $$\KR(R,\wedge_{\mathcal{V}})=\bigcap_{V\in \mathcal{V}}\KR(R,\wedge_{\{V\}})=\bigcap_{V\in \mathcal{V}}V_{M_V}(X)=\bigcap_{W\in\mathcal{W}}W=\KR(R,\star_a).$$ So that by Proposition \ref{Krr}, we have $F^{\star_a}=F^{\wedge_{\mathcal{V}}}$ for each $F\in f_h(R)$.
}
\end{remark}

Denote by $\Zar_h(R)$ the set of all gr-valuation overrings of $R$. We can endow $\Zar_h(R)$ with a Zariski topology, i.e., the topology having, as subbasic open subspaces, the subsets $\Zar_h(R[u])$ for $u$ varying in homogeneous elements of $R_H$.

\begin{proposition} The natural map $\psi:\Zar_h(R)\to\Zar(R), V\mapsto V_{M_V}$ is a topological embedding, where $M_V$ is the homogeneous maximal ideal of $V$.
\end{proposition}
\begin{proof} First, we show that $\psi$ is continuous. Let $\Zar(R[x])$ be a subbasic open subspace of $\Zar(R)$ for $x:=\frac{a}{b}\in K$. Let $a=\sum_{i=1}^{k}a_{i}$ and $b=\sum_{j=1}^{\ell}b_{j}$ be the decomposition of $a$ and $b$ into homogeneous components. Then we have the equalities
\begin{align*}
\psi^{-1}(\Zar(R[x]))=&\{V\in\Zar_h(R)\mid x\in V_{M_V}\}\\
=&\{V\in\Zar_h(R)\mid \inf\{v(a_i)\}\geq\inf\{v(b_j)\}\}\\
    =&\bigcup_{i,j}\{V\in\Zar_h(R)\mid v(a_{i})\leq v(a_{\lambda})\forall\lambda,v(b_{j})\leq v(b_{\mu})\forall\mu,v(b_{j})\leq v(a_{i})\} \\
    =&\bigcup_{i,j}\Zar_h\left(R[\{\frac{a_{\lambda}}{a_{i}}\mid\forall\lambda\}\cup\{\frac{b_{\mu}}{b_{j}}\mid\forall\mu\}\cup\{\frac{a_{i}}{b_{j}}\}]\right),
\end{align*}
which is an open subset of $\Zar_h(R)$, where $v$ is the valuation associated to the valuation domain $V_{M_V}$, thus $\psi$ is continuous.

Finally we show that the image via $\psi$ of an open set of $\Zar_h(R)$ is open in $\psi(\Zar_h(R))$. This is true since $\psi(\Zar_h(R[u]))=\psi(\Zar_h(R))\cap\Zar(R[u])$, for each homogeneous element $u\in R_H$.
\end{proof}

\begin{remark} Let $R=\bigoplus_{\alpha\in\Gamma}R_{\alpha}$ be a graded integral domain. The map $$\varphi:\Zar(\KR(R,b))\to\Zar_h(R), W\mapsto \bigoplus_{\alpha\in\langle\Gamma\rangle}(W\cap (R_H)_{\alpha})$$ is a bijection by Theorem \ref{*val}, with inverse given by $V\mapsto V_{M_V}(X)$, where $M_V$ is the homogeneous maximal ideal of $V$.
\end{remark}

The following result is a graded analogue of \cite[Lemma 1]{df86}.

\begin{theorem}\label{kr} The canonical bijection $\varphi:\Zar(\KR(R,b))\to\Zar_h(R)$, $W\mapsto \bigoplus_{\alpha\in\langle\Gamma\rangle}(W\cap (R_H)_{\alpha})$ is a homeomorphism, when they are endowed with the Zariski topologies.
\end{theorem}
\begin{proof} Let $\Zar_h(R[u])$ be a subbasic open subspace of $\Zar_h(R)$ for homogeneous element $u\in R_H$. Then we have the equality $$\varphi^{-1}(\Zar_h(R[u]))=\Zar(\KR(R,b)[u]).$$ It follows that $\varphi$ is continuous. Since $\varphi$ is a bijection, it now suffices to show that $\varphi$ is an open map. Let $\Zar(\KR(R,b)[\alpha])$ be a subbasic open subspace of $\Zar(\KR(R,b))$ for $0\neq\alpha\in K(X)$. Write $\alpha=\frac{\sum_{i=0}^{n}a_iX^i}{\sum_{i=0}^{m}b_iX^i}$, for $a_i,b_i\in R$. Let $a_i=\sum_{j=1}^{k_i}a_{ij}$ and $b_i=\sum_{j=1}^{\ell_i}b_{ij}$ be the decomposition of $a_i$ and $b_i$ into homogeneous components. For a gr-valuation overring $V$ with homogeneous maximal ideal $M_V$, let $v$ be the valuation on $K$, having valuation ring $V_{M_V}$. Then
\begin{align*}
    &\text{ }\varphi(\Zar(\KR(R,b)[\alpha]))=\{V\in\Zar_h(R)\mid\inf\{v(a_{ij})\}\geq\inf\{v(b_{ij})\}\}\\
    =&\bigcup_{p,q}\{V\in\Zar_h(R)\mid v(a_{ip})\leq v(a_{i\lambda})\forall\lambda,v(b_{iq})\leq v(b_{i\mu})\forall\mu,v(b_{iq})\leq v(a_{ip})\} \\
    =&\bigcup_{p,q}\Zar_h\left(R[\{\frac{a_{i\lambda}}{a_{ip}}\mid\forall\lambda\}\cup\{\frac{b_{i\mu}}{b_{iq}}\mid\forall\mu\}\cup\{\frac{a_{ip}}{b_{iq}}\}]\right),
\end{align*}
which is an open subset of $\Zar_h(R)$ completing the proof.
\end{proof}

A topological space $X$ is called a \emph{spectral} space if $X$ is quasi-compact (i.e. every open cover has a finite subcover) and $T_0$, the quasi-compact open subsets of $X$ are closed under finite intersection and form an open basis, and every nonempty irreducible closed subset of $X$ has a generic point (i.e. it is the closure of a unique point). By a theorem of Hochster, these are precisely the topological spaces which arise as the prime spectrum of a commutative ring endowed with the Zariski topology \cite{h69}. It is known that $\Zar(R)$ endowed with Zariski topology is a spectral space (\cite[Theorem 2]{df86} and \cite[Corollary 3.6]{ffl13}).

\begin{corollary} Let $R=\bigoplus_{\alpha\in\Gamma}R_{\alpha}$ be a graded integral domain. Then $\Zar_h(R)$ is homeomorphic to $\Spec(\KR(R,b))$. In particular $\Zar_h(R)$ is a spectral space.
\end{corollary}
\begin{proof} Since $\KR(R,b)$ is a B\'{e}zout domain, hence a Pr\"{u}fer domain by Proposition \ref{Krr}, $\Spec(\KR(R,b))$ is homeomorphic to $\Zar(\KR(R,b))$. Now Theorem \ref{kr} completes the proof.
\end{proof}

\section{The space of homogeneous preserving semistar operations of finite type on a graded domain}

Let $\Gamma$ be a (nonzero) torsionless commutative cancellative monoid, $R=\bigoplus_{\alpha\in\Gamma}R_{\alpha}$ be an integral domain graded by $\Gamma$, and $H$ be the set of nonzero homogeneous elements of $R$.

As in \cite{fs14}, the set $\SStar(R)$ of all semistar operations on $R$ was endowed with a topology (called the Zariski topology) having, as a subbasis of open sets, the sets of the form $$V_E:=\{\star\in\SStar(R)\mid 1\in E^{\star}\},$$ where $E\in\overline{\mathcal{F}}(R)$. This topology makes $\SStar(R)$ into a quasi-compact $T_0$ space. Consider $\SStar_f(R)$ as a topological space endowed with the subspace topology of the Zariski topology of $\SStar(R)$. 
It is shown that the set $\SStar_f(R)$, endowed with the Zariski topology, is a Spectral space \cite[Theorem 2.13]{fs14}.

We consider the set $\SStar_{hp}(R)$ (resp. $\SStar_{f,hp}(R)$) consisting of the set of homogeneous preserving (resp. finite type homogeneous preserving) semistar operations on $R$. The Zariski topology on $\SStar_{hp}(R)$ and $\SStar_{f,hp}(R)$ are just the subspace topology of the Zariski topology on $\SStar(R)$.

\begin{remark}{\em Let $R=\bigoplus_{\alpha\in\Gamma}R_{\alpha}$ be a graded integral domain.
\begin{itemize}
  \item[(a)] Since $\SStar(R)$ is a $T_0$-space by \cite[Proposition 2.4(1)]{fs14}, we have $\SStar_{hp}(R)$ and $\SStar_{f, hp}(R)$ are also $T_0$-spaces.
  \item[(b)] The subbasis of open sets of the Zariski topology on $\SStar_{hp}(R)$, are $V'_E:=V_E\cap\SStar_{hp}(R)$, where $E\in\overline{\mathcal{F}}(R)$.
  \item[(c)] The Zariski topology on $\SStar_{f, hp}(R)$ is determined by the finitely generated fractional ideals of $R$, in the sense that the collection of the sets of the form $W_F:=V_F\cap\SStar_{f, hp}(R)$, where $F$ varies among the finitely generated fractional ideals of $R$, is a subbasis, since there is an equality $$V_E\cap\SStar_{f, hp}(R)=\bigcup\{W_F\mid F\in f(R), F\subseteq E\}.$$
\end{itemize}}
\end{remark}

Recall that the set $\Over(R)$ of all overring of the integral domain $R$ can be endowed with the Zariski topology whose basic open sets are those of the form $B_F:=\Over(R[F])$, where $F$ ranges among the finite subsets of quotient field of $R$ \cite{o10}. Now, denote by $\Over_h(R)$ the set of homogeneous overrings of $R$. The Zariski topology on $\Over_h(R)$ is just the subspace topology of the Zariski topology on $\Over(R)$. For $T\in\Over_h(R)$, we have $\wedge_{\{T\}}\in\SStar_{f, hp}(R)$ such that $F\mapsto FT$ for any $F\in \overline{\mathcal{F}}(R)$, by Example \ref{ex}(e). Thus we can define $\imath:\Over_h(R)\to\SStar_{f, hp}(R)$, $T\mapsto\wedge_{\{T\}}$, which is injective.

\begin{proposition}\label{2.5} Endow $\Over_h(R)$ and $\SStar_{f,hp}(R)$ with Zariski topologies.
\begin{itemize}
  \item[(1)] The natural map $\imath:\Over_h(R)\to\SStar_{f, hp}(R)$, $T\mapsto\wedge_{\{T\}}$, is a topological embedding.
  \item[(2)] The map $\pi:\SStar_{f, hp}(R)\to\Over_h(R)$, defined by $\pi(\star):=R^{\star}$, for $\star\in\SStar_{f, hp}(R)$ is a continuous surjection.
  \item[(3)] $\pi\circ\imath$ is the identity map on $\Over_h(R)$, that is, $\pi$ is a topological retraction.
  \item[(4)] The canonical map $\varphi:\SStar_{hp}(R)\to\SStar_{f,hp}(R)$, $\star\mapsto\star_f$ is a topological retraction.
\end{itemize}
\end{proposition}
\begin{proof} Parts (1) and (4) are the same as proofs of \cite[Proposition 2.5, and 2.4(2)]{fs14}. Parts (2) and (3) follows from the equality $$\pi^{-1}(\Over_h(R[u]))=\SStar_{f,hp}(R)\cap W_{u^{-1}R},$$ see \cite[Proposition 2.1]{ffs16}.
\end{proof}

\begin{proposition}\label{2.7} Let $\Delta$ be a quasi-compact subspace of $\SStar_{f, hp}(R)$. Then, the semistar operation $\bigwedge(\Delta)$ belongs to $\SStar_{f, hp}(R)$.
\end{proposition}
\begin{proof} By Proposition \ref{2.5}(1), $\Delta$ is also a quasi-compact subspace of $\SStar_f(R)$. Hence by \cite[Proposition 2.7]{fs14}, we have that $\bigwedge(\Delta)$ is of finite type. Now by Example \ref{ex}(e), $\bigwedge(\Delta)$ is homogeneous preserving, that is $\bigwedge(\Delta)\in\SStar_{f, hp}(R)$.
\end{proof}

\begin{proposition}\label{2.11} Let $\{V'_{E_i}\mid i\in I\}$ be a nonempty family of subbasic open sets of the Zariski topology of $\SStar_{hp}(R)$. The following statements hold.
\begin{itemize}
  \item[(1)] $\bigcap\{V'_{E_i}\mid i\in I\}$ is a complete lattice (as a subset of the partially ordered set $(\SStar_{hp}(R),\leq)$).
  \item[(2)] $\bigcap\{V'_{E_i}\mid i\in I\}$ is a quasi-compact subspace of $\SStar_{hp}(R)$. In particular, $V'_E$ is quasi-compact for any $E\in \overline{\mathcal{F}}(R)$
\end{itemize}
\end{proposition}
\begin{proof} The proof is the same as proof of \cite[Proposition 2.11]{fs14}.
\end{proof}

\begin{lemma}\label{dd} Let $\emptyset\neq Y\subseteq\SStar_{f, hp}(R)$. Then $\bigvee(Y)\in\SStar_{f, hp}(R)$ and, for any $F\in \overline{\mathcal{F}}(R)$, we have
$$F^{\bigvee(Y)}=\bigcup\{F^{\sigma_1\circ\cdots\circ\sigma_n}\mid\sigma_1,\ldots,\sigma_n\in Y,n\in \mathbb{N}\}.$$
\end{lemma}

\begin{proof} It follows from \cite[Lemma 2.12]{fs14} (see also \cite[Page 1628]{aa90} for the case of star operations), that $\bigvee(Y)$ is a semistar operation of finite type, and that for any $F\in \overline{\mathcal{F}}(R)$, $F^{\bigvee(Y)}=\bigcup\{F^{\sigma_1\circ\cdots\circ\sigma_n}\mid\sigma_1,\ldots,\sigma_n\in Y,n\in \mathbb{N}\}$. Now it can be seen that $\bigvee(Y)\in\SStar_{f, hp}(R)$.
\end{proof}

Now we show $\SStar_{f,hp}(R)$ is a spectral space, whose proof is analogues to \cite[Thorem 2.13]{fs14}.

\begin{theorem}\label{spec} The space $\SStar_{f, hp}(R)$ of finite type homogeneous preserving semistar operation on $R$, endowed with the Zariski topology, is a spectral space.
\end{theorem}

\begin{proof} In order to prove that a topological space $X$ is a spectral space, we use the characterization given in \cite[Corollary 3.3]{f14}. We recall that if $\beta$ is a non empty family of subsets of $X$, for a given subset $Y$ of $X$ and an ultrafilter $\mathfrak{U}$ on $Y$, we set $$Y_{\beta}(\mathfrak{U}):=\{x\in X\mid[\text{for each }B\in\beta\text{, it happens that }x\in B\Leftrightarrow B\cap Y\in \mathfrak{U}]\}.$$
By \cite[Corollary 3.3]{f14}, for a topological space $X$ being a spectral space is equivalent to $X$ being a $T_0$-space having a subbasis $\mathcal{S}$ for the open sets such that $X_{\mathcal{S}}(\mathfrak{U})\neq\emptyset$, for every ultrafilter $\mathfrak{U}$ on $X$.

We know $X:=\SStar_{f, hp}(R)$ is a $T_0$-space, and set $\mathcal{S}:=\{W_F\mid F\in f(R)\}$ be the canonical subbasis of the Zariski topology on $X$. Let $\mathfrak{U}$ be an ultrafilter on $X$. It suffices to show the set $$X_{\mathcal{S}}(\mathfrak{U}):=\{\star\in X\mid[\text{for each }W_F\in\mathcal{S}\text{, it happens that } \star\in W_F\Leftrightarrow W_F\in\mathfrak{U}]\}$$ is nonempty. By Propositions \ref{2.7} and \ref{2.11}, any semistar operation of the form $\bigwedge(W_F)$ (where $F\in f(R)$) belongs to $\SStar_{f, hp}(R)$. Thus the semistar operation $$\star:=\bigvee(\{\bigwedge(W_F)\mid W_F\in\mathfrak{U}\})\in\SStar_{f, hp}(R)=X.$$ We claim that $\star\in X_{\mathcal{S}}(\mathfrak{U})$. Fix a finitely generated fractional ideal $F$ of $R$. It suffices to show $\star\in W_F\Leftrightarrow  W_F\in\mathfrak{U}$. First, assume $\star\in W_F$, i.e., $1\in F^{\star}$. By Lemma \ref{dd}, there exist finitely generated fractional ideals $F_1,\ldots,F_n$ of $R$ such that $1\in F^{\bigwedge(W_{F_1})\circ\cdots\circ\bigwedge(W_{F_n})}$ and $W_{F_i}\in \mathfrak{U}$, for any $i=1,\ldots,n$. Take a semistar operation $\sigma\in\cap_{i=1}^nW_{F_i}$. By definition, $\sigma\geq\bigwedge(W_{F_i})$, for $i=1,\ldots,n$, and thus $$1\in F^{\bigwedge(W_{F_1})\circ\cdots\circ\bigwedge(W_{F_n})}\subseteq F^{\sigma\circ\cdots\circ\sigma}=F^{\sigma},$$ thus $\sigma\in W_F$. This shows that $\cap_{i=1}^nW_{F_i}\subseteq W_F$ and thus, by definition of ultrafilter, $W_F\in\mathfrak{U}$, since $W_{F_1},\ldots,W_{F_n}\in\mathfrak{U}$. Conversely, assume that $W_F\in\mathfrak{U}$. This implies that $\bigwedge(W_F)\leq\star$. By definition, $1\in F^{\sigma}$, for each $\sigma\in W_F$, and thus $$1\in\cap_{\sigma\in W_F}F^{\sigma}=:F^{\bigwedge(W_F)}\subseteq F^{\star}.$$ This completes the proof.
\end{proof}

Given a spectral space $X$, the \emph{constructible topology} on $X$ is the coarsest topology such that every open and quasi-compact subset of $X$ (in the original topology) is both open and closed. By \cite{fl08}, the closed sets of the constructible topology in $X$ are the subsets $Y$ of $X$ such that, for every ultrafilter $\mathfrak{U}$ of $Y$, $$Y_{\beta}(\mathfrak{U}):=\{x\in X\mid[\text{for each }B\in\beta\text{, it happens that }x\in B\Leftrightarrow B\cap Y\in \mathfrak{U}]\}\subseteq Y,$$ where $\beta$ is the set of open and quasi-compact subspaces of $X$.

\begin{remark}\label{cons}{ The proof of Theorem \ref{spec} actually shows more than just the fact that $\SStar_{f, hp}(R)$ is a spectral space. In fact we also showed that  when $\SStar_{f}(R)$ is endowed with the constructible topology, $\SStar_{f, hp}(R)$ is a closed subspace.
}
\end{remark}

Let $\Spec(R)$ be the space of prime spectrum of $R$ endowed with Zariski topology, i.e. the topology whose closed sets are of the form $$V(\fa):=\{\fp\in\Spec(R)\mid \fp\supseteq\fa\},$$ for any ideal $\fa$ of $R$. The sets of the form $D(f):=\Spec(R)\setminus V(fR)$ form a basis of open sets for the Zariski topology. Denote by $\Spec_h(R)$ the subspace of $\Spec(R)$, of homogeneous prime ideals of $R$. We conclude the paper by showing that $\Spec_h(R)$ is a spectral space. For $a\in R$, let $V_h(a):=V(aR)\cap\Spec_h(R)$.

\begin{lemma}\label{clt} (cf. \cite[Lemma 2.4]{clt}) Let $Y$ be a subset of $X:=\Spec_h(R)$ and let $\mathfrak{U}$ be an ultrafilter on $Y$. Then $\fp:=\fp_{\mathfrak{U}}:=\{f\in H\mid V_h(f)\cap Y\in\mathfrak{U}\}$ is a homogeneous prime ideal of $R$.
\end{lemma}
\begin{proof} To prove that $\fp$ is homogeneous, assume that $a\in\fp$ and $a=\sum_{i=1}^n a_i$ is the decomposition of $a$ to homogeneous components. Then $V_h(a)\cap Y\in\mathfrak{U}$. Since $V_h(a)=V_h(\sum a_i)\subseteq V_h(a_i)$, we obtain that $V_h(a_i)\cap Y\in\mathfrak{U}$ for $i=1,\cdots,n$. Hence $a_i\in\fp$ for $i=1,\cdots,n$ and thus $\fp$ is a homogeneous ideals. Now let $fg\in\fp$ for $f,g\in H$. So that $(V_h(f)\cap Y)\cup(V_h(g)\cap Y)=V_h(fg)\cap Y\in\mathfrak{U}$. Then $V_h(f)\cap Y\in\mathfrak{U}$ or $V_h(g)\cap Y\in\mathfrak{U}$ by properties of ultrafilters. Thus $f\in\fp$ or $g\in\fp$. Therefore by \cite[Proposition 33, Page 124]{n68}, $\fp$ is a prime ideal of $R$.
\end{proof}

\begin{theorem}\label{s} The set $\Spec_h(R)$, endowed with the Zariski topology, is a spectral space.
\end{theorem}
\begin{proof} Let $X:=\Spec_h(R)$. We use \cite[Corollary 3.3]{f14} to prove that $X$ is a spectral space. Since $\Spec(R)$ is a $T_0$-space, we have $X$ is a $T_0$-space. The set $\mathcal{S}:=\{D_h(f):=\Spec_h(R)\setminus V(fR)\mid f\in H\}$ is a basis of open sets of $X$. Let $\mathfrak{U}$ be an ultrafilter on $X$. It suffices to show that the set $$X_{\mathcal{S}}(\mathfrak{U}):=\{\fp\in X\mid[\text{for each }D_h(f)\in\mathcal{S}\text{, it happens that } \fp\in D_h(f)\Leftrightarrow D_h(f)\in\mathfrak{U}]\}$$ is nonempty.

Let $\fp:=\{a\in H\mid V_h(a):=V(a)\cap\Spec_h(R)\in\mathfrak{U}\}$. Then $\fp$ is a homogeneous prime ideal of $R$ by Lemma \ref{clt}. We show that $\fp\in X_{\mathcal{S}}(\mathfrak{U})$. To this end we have to show that for each $f\in H$, $$\fp\in D_h(f)\Leftrightarrow D_h(f)\in\mathfrak{U}.$$ If $\fp\in D_h(f)$, then $\fp\notin V_h(f)$ and $f\notin\fp$. So that, $V_h(f)\notin\mathfrak{U}$, and hence $D_h(f)\in\mathfrak{U}$. Conversely, assume that $D_h(f)\in\mathfrak{U}$, Then $X\setminus V_h(f)\in\mathfrak{U}$, and so $V_h(f)\notin\mathfrak{U}$ (since otherwise $\emptyset\in\mathfrak{U}$ contradicts the definition of ultrafilters). Thus $f\notin\fp$ and hence $\fp\in D_h(f)$. This completes the proof.
\end{proof}

\begin{remark}{ Like Remark \ref{cons}, the proof of Theorem \ref{s} also shows that when $\Spec(R)$ is endowed with the constructible topology, $\Spec_h(R)$ is a closed subspace.
}
\end{remark}

\begin{center} {\bf ACKNOWLEDGMENT}

\end{center} I sincerely thank the referee for the careful reading of the manuscript, and several wonderful comments which greatly improved the paper and pointing me an error in an earlier version of this paper.


\end{document}